\documentclass[12pt]{amsart}
\usepackage{amsmath, latexsym, amsfonts, stmaryrd, amssymb, amsthm, amscd, array, caption, enumerate, array, color}
\numberwithin{equation}{section}
\newtheorem{theorem}{Theorem}[section]
\newtheorem{lemma}[theorem]{Lemma}

\renewenvironment{proof}[1][Proof]{\begin{trivlist}
\item[\hskip \labelsep {\bfseries #1}]}{\qed\end{trivlist}}

\newcommand{\bP}{{\ensuremath{\mathbf P}} }
\newcommand{\bE}{{\ensuremath{\mathbf E}} }

\newcommand{\mA}{{\ensuremath{\mathcal A}} }

\newcommand{\ga}{\alpha}

\newcommand{\gep}{\epsilon}       
\newcommand{\gp}{\varphi}

\newcommand{\gl}{\lambda}

\newcommand{\ind}{\mathbf{1}}

\newcommand{\Z}{\mathbb{Z}}

\renewcommand{\tilde}{\widetilde}
\renewcommand{\hat}{\widehat}

\newcommand{\simn}{\stackrel{n\to\infty}{\sim}}

\title[Renewal theory with no moments]{Local limit theorems and renewal theory with no moments}

\author[K. Alexander]{Kenneth S. Alexander}
\address{Department of Mathematics, KAP 108\\
University of Southern California\\
Los Angeles, CA  90089-2532 USA}

\author[Q. Berger]{Quentin Berger}
\address{LPMA, Universit\'e Pierre et Marie Curie\\
Campus Jussieu, case 188\\
4 place Jussieu, 75252 Paris Cedex 5, France}
\email{quentin.berger@upmc.fr}

\begin{document}

\begin{abstract}
We study i.i.d.~sums $\tau_k$ of nonnegative variables with index $0$: this means $\bP(\tau_1=n) = \gp(n) n^{-1}$, with $\gp(\cdot)$ slowly varying, so that $\bE(\tau_1^\gep)=\infty$ for all $\gep>0$.
We prove a local limit and  local (upward) large deviation theorem, giving the asymptotics of $\bP(\tau_k=n)$ when $n$ is at least the typical length of $\tau_k$.
A recent renewal theorem in \cite{cf:Nagaev} is an immediate consequence:
$\bP(n\in\tau) \sim \bP(\tau_1=n)/\bP(\tau_1 > n)^2$ as $n\to\infty$.
If instead we only assume regular variation of $\bP(n\in\tau)$ and slow variation of $U_n:= \sum_{k=0}^n \bP(k\in\tau)$, we obtain a similar equivalence but with $\bP(\tau_1=n)$ replaced by its average over a short interval.  We give an application to the local asymptotics of the distribution of the first intersection of two independent renewals.
We further derive downward moderate and large deviations estimates, that is, the asymptotics of $\bP(\tau_k \leq n)$ when $n$ is much smaller than the typical length of $\tau_k$.
\end{abstract}

\maketitle

\section{Introduction}

It is classical to study renewal processes $\tau = \{0=\tau_0,\tau_1,\tau_2,\dots\}$, 
and in particular the relation between the renewal mass function $\bP(n\in\tau)$ and the inter-arrival distribution $\bP(\tau_1=n)$. We assume the inter-arrival distribution $\bP(\tau_1=n)$ is regularly varying: there exists a positive slowly varying function $\gp(\cdot)$ and $\ga\geq 0$ such that
\begin{equation}
\label{eq:alpha}
 \bP(\tau_1=n) = \varphi(n)\, n^{-(1+\ga)}\, .
\end{equation}
In particular the process is aperiodic. 
The case receiving the least attention (under the general assumption \eqref{eq:alpha}) is $\ga=0$, in which $\tau_1$ has no moments and is not in the domain of attraction of a stable law, and that is our focus here. Tauberian theorems are of less use here than in other cases, so our methods are primarily probabilistic.  An example with $\alpha=0$ is the return times of symmetric simple random walk (SSRW) on $\Z^2$, $\tau=\{n\, , \, S_{2n}=0\}$, for which $\bP(\tau_1=n) \simn \pi/n(\log n)^2$, from \cite[Thm.~4]{cf:JP}.

\smallskip
The limiting distributions of $\tau_n$ and related quantities in the $\ga=0$ case have been studied in \cite{cf:Dar,cf:Gne,cf:Kas,cf:Nagaev,cf:NW,cf:Tei,cf:Wat}.
Defining $r(n):=\bP(\tau_1>n)$, Theorem~4.1 in \cite{cf:Dar} states that if $r(n)$ is slowly varying, then for any $y>0$
\begin{equation}
\label{Darling}
\bP\big( n\, r(\tau_n) <y \big) \to 1-e^{-y} \qquad \text{as } n\to+\infty .
\end{equation}
Recently in \cite{cf:Nagaev}, Nagaev proved a strong renewal theorem:
\begin{equation}\label{strongren}
  \bP(n\in\tau) \simn \frac{\bP(\tau_1=n)}{\bP(\tau_1 > n)^2},
\end{equation}
and for $\bP(\tau_k>n)$, some ``upward'' large deviation results (meaning for $n$ much larger than the typical size of $\tau_k$) were proved in \cite{cf:NW}.

\subsection{Renewal theorems}
\label{sec:renewals} The assumption \eqref{eq:alpha} is very natural: beyond the dimension-2 case, it includes the case $\tau=\{n\, , \, S_{2n}=0\}$, where $(S_n)_{n\geq 0}$ is SSRW on $\Z^d$ for any $d$. One has $\ga=1/2$ and $\gp(n) \stackrel{n\to\infty}{\to} (4\pi)^{-1/2}$ for $d=1$ (see e.g.\ \cite[Ch.~III]{cf:Feller}); and $\ga=\frac{d}{2}-1$, $\gp(n) \stackrel{n\to\infty}{\to} c_d$ for $d\geq 3$ (see \cite[Thm. 4]{cf:DK}).
Equation \eqref{eq:alpha} also includes the case $\tau=\{n\, , S_n=0\}$ where $(S_n)_{n\geq 0}$ is an aperiodic random walk in the domain of attraction of a symmetric stable law, see \cite[Thm. 8]{cf:Kesten}.

The asymptotics of the renewal function $\bP(n\in\tau)$ under \eqref{eq:alpha} have been widely studied in the literature, including \cite{cf:Doney}, \cite{cf:Eric}, \cite{cf:GL}, \cite{cf:Nagaev}, \cite{cf:Will}.
We recall briefly the results.

\smallskip
First, when $\tau$ is transient and \eqref{eq:alpha} holds, we have
\begin{equation}
\label{eq:transient}
\bP(n\in\tau) \simn \frac{\bP(\tau_1=n)}{\bP(\tau_1=+\infty)^2} \, .
\end{equation}
This is a consequence of Theorem~1 in \cite{cf:CNW}, and is also proved in \cite[App. A.5]{cf:Giac} with elementary methods.

\smallskip
If $\tau$ is recurrent, then
\begin{itemize}
\item if $\bE[\tau_1]<+\infty$, then the classical Renewal Theorem (see e.g.~\cite{cf:Asm}) gives that
\begin{equation}
\label{eq:renewthm}
\lim_{n\to\infty} \bP(n\in\tau) = \frac{1}{\bE[\tau_1]} \, ;
\end{equation}

\item if $\ga=1$ in \eqref{eq:alpha}, and $\bE[\tau_1]=+\infty$, Erickson \cite[Eq.~(2.4)]{cf:Eric} proved that
\begin{equation}
\label{eq:Eric}
\bP(n\in\tau) \simn  \frac{1}{\bE\left[ \tau_1\wedge n \right]} \, ;
\end{equation}

\item if $\ga\in (0,1)$ in \eqref{eq:alpha}, Doney \cite[Thm. B]{cf:Doney} proved that
\begin{equation}
\label{eq:Doney}
\bP(n\in\tau) \simn  \frac{\ga \sin(\pi \ga)}{\pi} \, n^{-(1-\ga)}\, \gp(n)^{-1} \, ;
\end{equation}

\item if $\ga=0$ in \eqref{eq:alpha}, then Nagaev \cite{cf:Nagaev} showed
\begin{equation}
\label{eq:Nagaev}
\bP(n\in\tau) \simn \frac{\bP(\tau_1=n)}{\bP(\tau_1 > n)^2} \, .
\end{equation}
\end{itemize}

\medskip
The condition \eqref{eq:alpha} is not best possible for the validity of these strong renewal theorems with infinite mean. 
Assume simply that $\bP(\tau_1>n) \simn \ga^{-1} \gp(n) n^{-\ga}$ with $\ga\in(0,1]$ (and $\bE[\tau_1]=+\infty$ if $\ga=1$), so that $\tau_1$ is in the domain of attraction of a stable law with index $\ga$.
Garsia and Lamperti \cite{cf:GL} showed that \eqref{eq:Doney} holds whenever $\ga\in(\tfrac12,1)$, and Erickson proved \eqref{eq:Eric} in the case $\ga=1$.
When $\ga\in(0,\tfrac12]$, some additional conditions on the distribution of $\tau_1$ are necessary for \eqref{eq:Doney} to be valid, and sufficient ones were given in \cite{cf:Chi}, \cite{cf:Chi2}, \cite{cf:Doney}, \cite{cf:Will}. 
It is only recently that a complete necessary and sufficient condition for the strong renewal theorem \eqref{eq:Doney} was proved in simultaneous papers by Caravenna \cite{cf:C15} and Doney \cite{cf:D15}.
A necessary and sufficient condition remains to be found in the case $\ga=0$.
\smallskip

Throughout the paper, $c_1,c_2,\dots$ are constants depending only on the distribution of $\tau_1$. 
Also, we treat certain large quantities at times as if they were integers, to avoid the clutter of integer-part notation; in all cases these can be treated as if the integer-part notation were in use.

\medskip
Our first result is a local limit and local (upward) large deviation theorem, proved in Section \ref{sec:thm1}, in the case of a recurrent $\tau$. 
Define $r_n:=r(n):=\bP(\tau_1> n)$, which in the $\ga=0$ case is slowly varying and satisfies (see  \cite[Proposition 1.5.9a]{cf:BGT})
\begin{equation}\label{phinrn}
 \gp(n)= o(r_n) \quad \text{as } n\to\infty.
\end{equation}
In particular we have $\gp(n)\to 0$.

In \cite{cf:NW}, it is proved that $\bP(\tau_k > n) \sim k r_n $ as $n,k\to\infty$ with $kr_n\to 0$.
We improve here  this result by establishing a local limit theorem, and extending the range of validity to $k\gp(n)\to 0$. This extension is significant because $kr_n\to 0$ allows only values of $n$ much larger than the typical value of $\tau_k$; see the remarks following the theorem.

\begin{theorem}
\label{prop:locallimit}
If $\tau$ is recurrent and \eqref{eq:alpha} holds with $\ga=0$, 
then uniformly for $k$ such that $k \varphi(n) \to 0$, we have 
\begin{equation}
\label{locallimit}
  \bP(\tau_k = n) \simn k\bP(\tau_1=n) (1-r_n )^k .
\end{equation}
Further, there exists a constant $c_1>0$ such that for $n$ sufficiently large and all $1\leq k\leq n$,
\begin{equation}
\label{uniformbound}
  \bP(\tau_k = n) \leq c_1 k\bP(\tau_1=n) (1-r_n )^k .
\end{equation}
\end{theorem}

Note that, as soon as $k \gg r_n^{-1}$, we have $\bP(\tau_k \leq n)\leq (1-r_n)^k \to 0$, and $n$ is therefore much smaller than the typical size of $\tau_k$. By \eqref{phinrn}, $k \gg r_n^{-1}$ is consistent with the hypothesis $k \ll 1/\varphi(n)$. 
Equation \eqref{locallimit} therefore includes $n$ down to a size much smaller than the typical size of $\tau_k$. 
 Heuristically, \eqref{locallimit} says that even for much smaller-than-usual $n$, when $\tau_k=n$ it is because there was a single gap of length very close to $n$, among the first $k$ gaps $\tau_j-\tau_{j-1}$; this is unique to $\ga=0$.

In comparison, in the case where \eqref{eq:alpha} holds with $\ga\in(0,1)$,  Doney \cite[Thm.~A]{cf:Doney} proved that $\bP(\tau_k=n) \sim k \bP(\tau_1=n)$ provided that $k r_n \to 0$, and Denisov, Dieker and Shneer \cite[Section 9]{cf:DDS08} proved a similar, more general, local large deviation theorem that applies in the $\ga>0$ case. If we consider the case $k r_n\to x \in (0,+\infty)$, we have that  $n/a_k\to x^{-1/\ga}$, where $a_k$ is such that $\bP(\tau_1>a_k)\sim k^{-1}$ (so that $\tau_k/a_k$ converges to an $\ga$-stable distribution with non-degenerate density $g$). Then, Gnedenko's local limit theorem (see \cite[\S~50]{cf:GK} gives that $\bP(\tau_k =n) \sim  \ga^{-1} x^{-(1+1/\ga)}g(x^{-1/\ga}) k \bP(\tau_1=n) $, in contrast with \eqref{locallimit} when $\ga=0$.

\medskip
The strong renewal theorem \eqref{strongren} from \cite{cf:Nagaev} is an easy consequence of Theorem \ref{prop:locallimit}, as follows.
Assume \eqref{eq:alpha} with $\ga=0$ and let $\theta_n$ satisfy $r_n^{-1} \ll \theta_n\ll \gp(n)^{-1}$. We write $\bP(n\in\tau)= \sum_{k=1}^n \bP(\tau_k =n)$, and decompose it according to whether $k$ is smaller or larger than $\theta_n$.
Thanks to \eqref{locallimit}, by our choice of $\theta_n$ we have 
\[
  \sum_{k\leq \theta_n} \bP(\tau_k =n) \simn \sum_{k\leq \theta_n} k\bP(\tau_1=n)(1-r_n)^k \ \simn \  r_n^{-2}\bP(\tau_1=n) \, .
\]
For the rest of the sum, we use \eqref{uniformbound} together with $\theta_n  \gg r_n^{-1}$, to get that, for $n\geq n_0$
\[\sum_{k>\theta_n} \bP(\tau_k=n)  \leq c_1 \sum_{k>\theta_n} k\bP(\tau_1=n)(1-r_n)^k =o(1) r_n^{-2}\bP(\tau_1=n) \, \quad \text{as } n\to+\infty.\]
These two estimates give \eqref{strongren}.
Combining with \eqref{eq:transient}, we obtain the following statement:
if  \eqref{eq:alpha} holds with $\ga\geq 0$, and $\bP(\tau_1>n)$ is slowly varying (that is, either $\tau$ is transient, or $\tau$ is recurrent with $\ga=0$), then \eqref{strongren} holds.

The heuristic behind \eqref{strongren} may be seen by restating it as $\bP(\tau_1=n\mid n\in\tau) \sim \bP(\tau_1>n)^2$.  This says that given $n\in\tau$, in order to have $\tau_1=n$ (i.e.~no renewals between 0 and $n$), the trajectory mainly needs to ``escape'' without renewals at each end, and these two escapes are approximately independent, each with probability near $\bP(\tau_1>n)$.  This independence in the recurrent case is unique to $\ga=0$, since in that case the only renewals that typically occur given $n\in\tau$ are very close to 0 and $n$.


\subsection{Large and moderate downward deviations}

Theorem \ref{prop:locallimit} may be viewed as both a local limit theorem and a local large deviation theorem for the case $\ga=0$, covering upward deviations (in the sense that $n$ is much larger than the typical size of $\tau_k$) and downward deviations that are not too great. 
As a complement we now consider estimates for downward deviations of the form $\bP(\tau_k \leq n)$ for $n$ much smaller than the typical size of $\tau_k$, that is $k r_n \to \infty$. 

Let $\gp^*$ denote a slowly varying function conjugate to $\gp$, that is, such that $x\mapsto x\gp^*(x)$ is an asymptotic inverse of $y\mapsto y\gp(y)$, see \cite[\S 1.5.7]{cf:BGT} for more. For most common slowly varying functions $\gp$ one has $\gp^*\sim 1/\gp$, but this is not true if $\gp$ is ``barely slowly varying,'' for example $\gp(n) = n^{1/\log\log n}$.
We will prove the following in Section \ref{largedev}.

\begin{theorem}\label{largedev}
Suppose $\tau$ is recurrent and \eqref{eq:alpha} holds with $\ga=0$. Let $n\geq k$.

(i) Given $M>0$ there exists $a_M$, with $a_M\to 1$ as $M\to 0$, such that 
if $n$ is large and $k\gp(n) \leq M$, then 
\begin{equation}
\label{extDarling}
  a_M (1-r_{n})^k \leq \bP(\tau_k \leq n) \leq (1-r_{n})^k .
\end{equation}

(ii) If $k,n \to+\infty$ with $k\gp(n) \to +\infty$ and $n/k \to + \infty$, then we have
\begin{equation}\label{phistarrate}
  \bP(\tau_k \leq n) = 
    \exp\left\{ -(1+o(1))\, k\, r\left( \frac{n}{k} \gp^*\left(\frac{n}{k}\right) \right) \right\}.
\end{equation}

(iii) For $n=b k$ with $b\geq 1$, the limit
$  -I(b) = \lim \limits_{n\to\infty} \frac{1}{n}\log \bP(\tau_k \leq bk)$
exists, and it is finite if $b\geq \min\{j:\bP(\tau_1=j)>0\}$. Moreover, it satisfies
\[
  I(b)  \sim r(b\gp^*(b)) \quad \text{as } b\to+\infty.
\]

\end{theorem}

This theorem extends the result \eqref{Darling} of Darling \cite{cf:Dar} to the case $y\to+\infty$ as $n\to+\infty$. In particular, (i) allows to recover \eqref{Darling} by taking $k=y / r_n$ (since $\{\tau_k \leq n\} =\{r(\tau_k) \geq  y/k\}$), and moreover extends it to $\bP(k\, r(\tau_k) \geq y) \sim e^{-y}$ as $k\to\infty$, uniformly for $y \ll r_n /\gp(n)$ (we recall~\eqref{phinrn}).


\subsection{Reverse renewal theorems.}

Though \eqref{eq:alpha} is very natural, verifying that it holds is often difficult, for example if $\tau=\{n,\, S_n=0\}$, with $(S_n)_{n\geq 0}$ an aperiodic random walk in the domain of attraction of a symmetric stable distribution, see \cite{cf:JP}. But in that case, a local limit theorem (see \cite[\S~50]{cf:GK}) easily gives the asymptotic behavior of $\bP(S_n=0)=\bP(n\in\tau)$. Therefore, one would like to get a general result to infer from $\bP(n\in\tau)$ something about the behavior of $\bP(\tau_1=n)$.  We call such a result a \emph{reverse renewal theorem}. An additional application of such theorems is given in Section \ref{sec:applications}.

In general, it is not true that regular variation of $\bP(n\in\tau)$ implies regular variation of $\bP(\tau_1=n)$, an example being given in Section \ref{sec:example}.
But the average of the values $\bP(\tau_1=n)$ over a relatively short interval may be better behaved.  In fact we can obtain a reverse renewal theorem corresponding to \eqref{strongren} and \eqref{eq:transient} in the $\alpha=0$ case, as follows.

\smallskip
Define 
\[  
  U_n:= \sum_{k=0}^n \bP(k\in\tau), \qquad U_{\infty}:= \bE[|\tau|]=\sum_{k=0}^{\infty} \bP(k\in\tau) \quad \left(= \frac{1}{\bP(\tau_1=\infty)}\ \text{if }
   U_\infty<\infty\right),
  \]
and note that 
\begin{equation} \label{Un}
  \text{if $U_n$ is slowly varying, then } \ U_n \simn \bP(\tau_1>n)^{-1}.
  \end{equation}
This is trivial if $\tau$ is transient: $|\tau|$ is then a geometric random variable, and $U_n$ converges to $\bE[|\tau|]=\bP(\tau_1=+\infty)^{-1}$. In the recurrent case, we refer to Theorem 8.7.3 in \cite{cf:BGT}; the proof uses standard properties of convolution of Laplace transforms. Note that in the following we do not assume \eqref{eq:alpha}.

\begin{theorem}
\label{thm2}
Assume that $\bP(n\in\tau)$ is regularly varying and $U_n$ is slowly varying.
Then there exist $\gep_n\to 0$ such that 
\begin{equation}\label{summed}
  \frac{1}{ \gep_n n}  \sum_{(1-\gep_n) n < k \leq n} \bP(\tau_1=k) \simn \bP(\tau_1>n)^2 \bP(n\in\tau)  \, .
\end{equation}
If also $\bP(\tau_1=n)$ is regularly varying, then 
\begin{equation}\label{nosum}
  \bP(\tau_1=n) \simn \bP(\tau_1>n)^2 \bP(n\in\tau).
\end{equation}
\end{theorem}

This theorem applies in the recurrent case when $\bP(n\in\tau)$ is regularly varying with index $-1$, and in the case of a transient renewal~$\tau$.
When $\tau$ is transient, we are able to prove the following stronger statement.

\begin{theorem}
\label{thm:transient}
If $\bP(n\in\tau)$ is regularly varying and $\tau$ is transient, then
\[\bP(\tau_1=n) \simn   \bP(\tau_1=\infty)^2 \bP(n\in\tau) \, .\]
\end{theorem}

This theorem was proved in \cite{cf:DK} in the case where $\tau_1,\tau_2,\cdots$ are the return times to the origin of a transient aperiodic random walk,  and can be proved via Banach Algebra techniques, using \cite[Theorem 1]{cf:CNW} as suggested in \cite[Section 2]{cf:DK}. However, we give here an elementary probabilistic proof.

Section \ref{sec:transient} is devoted to the proof of Theorem \ref{thm:transient}, and Section \ref{sec:thm2} to the proof of Theorem~\ref{thm2}. Finally, in Section~\ref{sec:example}, we give an example where $\bP(\tau_1=n)$ is not regularly varying but $\bP(n\in\tau)$ is, and $U_n$ is slowly varying. This shows that  \eqref{nosum} cannot hold in the general case of a recurrent renewal, and our Theorem \ref{thm2} is in that sense optimal.

In general, Theorem \ref{thm2} reduces the problem of proving \eqref{nosum} to showing that $\bP(\tau_1=k)$ is approximately constant over the interval $((1-\gep_n)n,n]$.


\subsection{Application of reverse renewal theorems: the  intersection of two independent renewals}
\label{sec:applications}

Let $\tau$ and $\sigma$ be independent renewal processes with inter-arrival distributions satisfying
\begin{equation}
\label{alphas}
\bP(\tau_1 = n) = \gp(n) n^{-(1+\ga)} \, , \qquad \bP(\sigma_1=n) = \tilde\gp(n) n^{-(1+\tilde\ga)}
\end{equation}
for some $\ga,\tilde\ga\geq 0$ and slowly varying functions $\gp(\cdot), \tilde\gp(\cdot)$. We assume $\ga\leq \tilde\ga$.

We denote the intersection $\rho:=\tau\cap \sigma$, which is a renewal process
with renewal mass function and renewal function
\[
  \bP(n\in\rho) = \bP(n\in\tau) \bP(n\in\sigma), \quad U_n^* = \sum_{k=0}^n \bP(k\in\rho).
  \]
These are regularly varying, and their asymptotic behavior is thus known from the results for $\sigma,\tau$ in Section \ref{sec:renewals}.  In \cite{cf:ABintersect} our reverse renewal theorems, \ref{thm2} and \ref{thm:transient}, are applied to help establish the following.  If $\rho$ is transient (i.e.~$U_\infty^*<\infty$) then
\[
 \bP(\rho_1=n) \simn (U_\infty^*)^{-2} \bP(n\in\tau)\bP(n\in\sigma).
\]
If $\rho$ is recurrent and either (i) $\ga,\tilde\ga\in (0,1)$ with $\ga+\tilde\ga=1$, or (ii) $\ga=0,\tilde\ga\geq 1$, then $U_n^*$ is slowly varying, and
\begin{equation}\label{summedA}
 \bP(\rho_1=n) \simn (U_n^*)^{-2} \bP(n\in\tau)\bP(n\in\sigma) \simn \frac{\psi^*(n)}{n}
\end{equation}
for some (asymptotically known) slowly varying $\psi^*$.  In \cite{cf:ABintersect}, general $0\leq\ga\leq\tilde\ga$ are covered, and Theorems  \ref{thm2} and \ref{thm:transient} here are essential for the cases (i) and (ii).
The key step to get from \eqref{summed} for $\rho$ to \eqref{summedA} is to show that, due to the regularity \eqref{alphas} in $\sigma$ and $\tau$, $\bP(\rho_1=k)$ is approximately constant over short intervals, so that the left side of \eqref{summed} (for $\rho$) is asymptotic to $\bP(\rho_1=n)$.\\


\section{Proof of Theorem \ref{prop:locallimit}}
\label{sec:thm1}

We first prove \eqref{locallimit}, and turn to \eqref{uniformbound} as a second step. We introduce some notations: let
\[G_i:=\tau_i-\tau_{i-1} \, \quad \text{ and }\quad  M_k:= \max_{1\leq i\leq k} G_i \, .\]
We also let $\hat{G}^{(m)}_1,\dots,\hat{G}^{(m)}_k$ be i.i.d.\ with distribution $\bP(\tau_1\in\cdot\mid \tau_1\leq m)$.


\subsection{Proof of the local limit and local large deviation result \eqref{locallimit}}
The proof is divided into three steps, in which we control several contributions to $\bP(\tau_k=n)$.
\begin{itemize}
\item Step 1. Contribution of the case of only one jump larger than $(1-\gep) n$, all the other ones being (necessarily) smaller than $n/2$. This gives the right order in Theorem \ref{prop:locallimit} when $k\ll \gp(n)^{-1}$;
\item Step 2. Contribution of the case when all jumps are smaller than $n/2$: it is negligible, so there must be one jump larger than $n/2$ (and there can be only one such jump);
\item Step 3. Contribution of the case when there is one jump larger than $n/2$, but smaller than $(1-\gep) n$. This is also negligible.
\end{itemize}

\medskip
{\bf Step 1:} We show that, for any fixed $\gep>0$, and provided that $k\gp(n) \stackrel{n\to\infty}{\to} 0$,
\begin{equation}\label{onebig}
  \bP\Big(\tau_k=n, M_k > (1-\epsilon) n\Big) = (1+O(\gep))\, k\bP(\tau_1=n)(1-r_n)^k, \quad \text{as } n\to\infty.
\end{equation}

\smallskip
We have
\begin{equation}\label{bigsum2}
  \bP \Big(\tau_k=n, M_k > (1-\epsilon) n \Big) = k (1-r_n)^{k-1}\sum_{m =1}^{ \epsilon n} \bP\left(\sum_{i=1}^{k-1} \hat{G}^{(n)}_i = m\right) \bP(\tau_1=n-m).
\end{equation}
This gives the upper bound 
\begin{align}
\label{upperbd}
  \bP\left(\tau_k=n, M_k > (1-\epsilon) n\right)&\leq k (1-r_n)^{k-1} \max_{(1-\epsilon)n\leq j\leq n} \bP(\tau_1=j) \notag\\
  & \leq (1+ 2\gep) k (1-r_n)^{k} \bP(\tau_1=n), 
\end{align}
provided that $n$ is large enough.

In the other direction, \eqref{bigsum2} gives
\[ \bP \Big( \tau_k=n, M_k > (1-\epsilon) n \Big) \geq k (1-r_n)^{k-1} \bP\left(\sum_{i=1}^{k-1} \hat{G}^{(n)}_i \leq \epsilon n \right) \min_{(1-\epsilon)n\leq j\leq n} \bP(\tau_1=j) \, .\]
Then, using that $\bE[\hat{G}^{(n)}_1] = (1-r_n)^{-1} \sum_{x=1}^n \gp(x) \simn  n \gp(n)$, we have that  for $n$ large enough
\[\bP\left(\sum_{i=1}^{k-1} \hat{G}^{(n)}_i \leq \epsilon n \right)  \geq  1- \frac{\bE[\hat{G}^{(n)}_1]}{\gep n} \geq 1-\frac{2(k-1)\gp(n)}{\gep} \, .\]
Therefore, since $k\gp(n)\to 0$, we end up with
\begin{equation}
\label{lowerbd}
  \bP \Big( \tau_k=n, M_k > (1-\epsilon) n \Big) \geq (1-2\epsilon) k (1-r_{n})^{k}   \bP(\tau_1=n).
\end{equation}
provided that $n$ is large enough.

\medskip
{\bf Step 2:} We want to show that the main contribution to $\bP(\tau_k=n)$ comes when $M_k\geq  n/2$.
We prove that there exists a constant $c_2>0$ such that, if $k\gp(n)$ is small enough,
\begin{equation}\label{smallmax}
  \bP\left(\tau_k=n,M_k\leq n/2\right) \leq c_2 k^2\gp(n)\,  \bP(\tau_1=n)(1-r_n)^k \, ,
\end{equation}
which is negligible compared to \eqref{onebig} when $k\gp(n)\to 0$.
It is sufficient to show that, if $k\gp(n)$ is small enough,
\begin{equation}\label{suff}
  \bP\left(\sum_{i=1}^{k} \hat{G}^{(n)}_i = n \, ; \, \hat{G}^{(n)}_i \leq n/2 \text{ for all } i \leq k \right) \leq c_2 k^2\gp(n)\, \bP(\tau_1=n)
\end{equation}

To prove this, we rely on the following lemma,  which is a special case of the Fuk-Nagaev inequality, see \cite[Theorem 1.1]{cf:NagaevOverview}. We include a proof here since it is short and elementary for our case.

\begin{lemma}
\label{lem:nottoolarge}
Suppose \eqref{eq:alpha} holds with $\ga=0$. There exist constants $c_3,c_4>0$ such that for $n$ large, for all $1\leq m\leq n$ and $k\geq 0$,
\begin{equation}
  \bP\left( \sum_{i=1}^k \hat{G}_i^{(m)} \geq n/2 \right) \leq \left(  \frac{c_3\, k m\gp(m) }{n}\right)^{\frac{n}{2m}}  \leq \Big(c_4 k \gp(n) \Big)^{\frac{n}{2m}}\, .
\end{equation}
\end{lemma}

\begin{proof}
The second inequality is a consequence of the fact that $m\gp(m)$ is asymptotically increasing, so we prove the first inequality.

For any $\lambda>0$ we have
\begin{equation}\label{lambdabound}
  \bP \Big( \sum_{i=1}^k \hat{G}_i^{(m)} \geq n/2 \Big) \leq e^{-\lambda n/2} \bE\left[ e^{\lambda \hat{G}_1^{(m)}}\right]^k \, .
  \end{equation}
There exists a constant $c_5$ such that for any $j\geq 1$
\begin{equation}
\label{EGm}
\bE \left[(\hat{G}_1^{(m)})^j \right]  \leq m^{j-1} \, \bE[\tau_1 \mid \tau_1\leq m ]  \leq c_5  m^{j} \gp(m)  \, .
\end{equation}
Hence
\[\bE\left[ e^{\lambda \hat{G}_1^{(m)}}\right] \leq 1+ c_5 \gp(m)\big( e^{m \lambda} -1\big) \, .\]

Now, let us define $\lambda$ by
\[
  c_5 \gp(m)\big( e^{m \lambda} -1\big) = \frac{n}{km},
\]
so that
\[
  \bE\left[ e^{\lambda \hat{G}_1^{(m)}}\right]^k \leq e^{n/m}
\]
and
\[ e^{-\lambda n/2}  = \left( 1 + \frac{n}{c_5 m k \varphi(m)} \right)^{-n/2m} \leq \left( \frac{c_5 k m\gp(m)}{n}\right)^{n/2m}\, . \]
Therefore, \eqref{lambdabound} yields
\begin{equation}\label{lambdabound2}
  \bP \Big( \sum_{i=1}^k \hat{G}_i^{(m)} \geq n/2 \Big)  \leq \left(  \frac{c_5 k m\gp(m)}{n}\right)^{n/2m} e^{n/m} \leq \Big(   \frac{c_5 e^2 k m\gp(m)}{n} \Big)^{\frac{n}{2m}} \, .
\end{equation}
\end{proof}

To control the probability on the left in \eqref{suff}, we decompose it according to the value of the largest $\hat{G}^{(n)}_i$.
Let us denote
$m_s:=2^{-s} n $ and $J_s = (m_{s+1},m_s].$  
We have
\begin{align}\label{sdecomp}
\bP \Big( \sum_{i=1}^k & \hat{G}^{(n)}_i = n \, ;\, \hat{G}^{(n)}_i \leq n/2 \text{ for all } i\leq k \Big)\notag\\
 &  = \sum_{n/k\leq m\leq n/2} k\, \bP\left( \hat{G}^{(n)}_1=m,\, \hat{G}^{(n)}_i \leq m \ \text{for all} \, 2\leq i \leq k\, , 
    \sum_{i=2}^k \hat{G}^{(n)}_i =n-m \right) \notag\\
  &\leq \sum_{1\leq s\leq \log_2 k}\ \sum_{m\in J_s}  
    k \left( \frac{1-r_{m}}{1-r_n} \right)^{k-1} 
    \bP\left( \hat{G}^{(n)}_1 =m \right) \bP\left( \sum_{i=2}^k \hat{G}_i^{(m)} =n-m \right) \notag\\
  & \leq 2k\sum_{1\leq s\leq \log_2 k}\ \sum_{m\in J_s} \frac{\varphi(m)}{m} 
    \bP\left( \sum_{i=1}^{k-1} \hat{G}_i^{(m)} =n-m \right) \notag \\
  &\leq 2c_6 k \sum_{1\leq s\leq \log_2 k} \frac{ \varphi(m_{s+1}) }{ m_{s+1} } 
    \bP\left( \sum_{i=1}^{k-1} \hat{G}_i^{(m_s)} \geq \frac{n}{2} \right),
\end{align}
where in the last inequality we used that there exists $c_6$ such that for sufficiently large $m_s$ and all $m\in J_s$, $\gp(m)\leq c_6 \, \gp(m_{s+1})$.  Since $n/k \gg n\gp(n)\to\infty$, all values $m_s$ in \eqref{sdecomp} are sufficiently large in this sense, when $n$ is large.

Since $\gp$ is slowly varying, given $a\leq 1$ we have $\gp(an)/\gp(n) \leq 1/a$ for $n$ large.  With \eqref{sdecomp} and Lemma \ref{lem:nottoolarge} this shows that
\begin{align}
\bP\left(\sum_{i=1}^{k} \hat{G}^{(n)}_i = n \, ; \, \hat{G}^{(n)}_i \leq n/2 \text{ for all } i \leq k \right) 
  &\leq 2 c_6 k \sum_{s\geq 1} \frac{2^{s+1} \gp(n)}{2^{-(s+1)} n}   \Big(  c_4\,  k \gp(n) \Big)^{\frac{n}{2m_s}}  \notag\\
& \leq 8 c_6\,  k\,  \frac{\gp(n)}{n} \sum_{s\geq 1} 4^{s}  \Big(  c_4\,  k \gp(n) \Big)^{2^{s-1}} \notag\\
&\leq  c_2k^2\gp(n) \bP(\tau_1=n) \, ,
\end{align}
where we used in the last inequality that $k\gp(n)$ is small.
Hence, \eqref{suff} is proved, and so is \eqref{smallmax}.

\medskip
{\bf Step 3:}  We show that the main contribution to $\bP(\tau_k=n)$ comes when not only $M_k\geq n/2$, but when $M_k \geq (1-\epsilon) n$: we prove that for $n$ large enough,
\begin{equation}\label{bigsum}
  \bP \Big(\tau_k=n,  n/2 < M_k \leq (1-\epsilon) n \Big) \leq  \frac{6}{\gep}k^2\gp(n) \bP(\tau_1=n)\, (1-r_n)^k.
\end{equation}
Indeed, we have that
\[
\bP \Big(\tau_k=n,\,   n/2 < M_k \leq (1-\epsilon) n \Big) \leq k\,(1-r_n)^{k-1} \max_{n/2\leq j\leq n} \bP( \tau_1 =j) \bP\left( \sum_{i=1}^{k-1} \hat{G}^{(n)}_i  \geq \gep n\right)\, .
\]
Then, we use that $\max_{n/2\leq j\leq n} \bP( \tau_1 =j) \leq 3 \bP(\tau_1=n)$ provided that $n$ is large enough, together with Markov's inequality and the fact that $\bE[\hat{G}^{(n)}_1] \leq 2 n\gp(n)$ when $n$ is large enough. This yields \eqref{bigsum}.

Combining
\eqref{upperbd}-\eqref{lowerbd} with \eqref{smallmax} and \eqref{bigsum}, since $\gep$ is arbitrary we get that, uniformly for $k$ such that $k\gp(n)\to 0$, \eqref{locallimit} holds.\qed


\subsection{Proof of the uniform bound \eqref{uniformbound}}

To prove the uniform bound, we rely on Lemma \ref{lem:nottoolarge}, and we decompose the probability according to the value of $M_k$.

Let $n\geq n_0$ and define
\[ \ell_n=\min\{\ell:2^\ell\geq n\}, \quad \ell_{n,k}:= \max\{\ell :  c_3k 2^{\ell} \gp(2^{\ell}) \leq \tfrac12 n \},
  \]
where $c_3$ is the constant from Lemma \ref{lem:nottoolarge}.

Then for some (large) $\ell_0$, there exists a constant $c_7>0$ such that for all $\ell_0<\ell\leq\ell_n-2$,
\begin{align}\label{maxbound}
  \bP\big(\tau_k = n,\, & M_k \in (2^{\ell-1},2^{\ell}] \big) \leq k\bP\left( G_1\in (2^{\ell-1},2^{\ell}], \max_{2\leq i\leq k} G_i \leq 2^\ell,\tau_k=n\right) \notag\\
  &\leq k (1-r_{2^\ell})^{k-1} \max_{m\in (2^{\ell-1},2^{\ell}]} \bP\left(\tau_1=m \right)
    \bP\left( \sum_{i=1}^{k-1} \hat{G}_i^{(2^{\ell})} \in (n-2^{\ell}, n] \right) \notag\\
  &\leq c_{7}k (1-r_{2^\ell})^{k-1} \frac{\gp(2^{\ell})}{2^{\ell}} \bP\left( \sum_{i=1}^{k-1} \hat{G}_i^{(2^{\ell})} > \frac{n}{2} \right).
\end{align}
We now have 4 cases according to the value of $\ell$.

\medskip
{\bf Case 1.} For $\ell_0\vee\ell_{n,k}<\ell\leq \ell_n-2$ we bound the last probability in \eqref{maxbound} by $1$, and observe that provided $\ell_0$ is large enough, $r_{2^\ell} - r_n \geq \tfrac12 \gp(2^\ell)$, which leads to
\begin{align}\label{bigell}
  \bP\big(\tau_k = n,2^{\ell_0\vee\ell_{n,k}} < M_k \leq 2^{\ell_n-2} \big) &\leq 2c_{7}\, k(1-r_n)^k \sum_{\ell=\ell_0\vee\ell_{n,k}+1}^{\ell_n-2} \frac{\gp(2^{\ell})}{2^{\ell}}
    \left( 1 - \frac{r_{2^\ell} - r_n}{1-r_n} \right)^k \notag\\
  &\leq 2 c_7 \, k(1-r_n)^k \sum_{\ell=\ell_0\vee\ell_{n,k}+1}^{\ell_n-2} \frac{\gp(2^{\ell})}{2^{\ell}} e^{-k\gp(2^\ell)/4} \notag\\
  & \leq 2 c_7 \, k(1-r_n)^k \frac{\gp(n)}{n} \sum_{\ell=1}^{\ell_n-2} \frac{n}{2^{\ell}}\frac{\gp(2^{\ell})}{\gp(n)} e^{-n/8 c_3 2^{\ell}}\, ,
\end{align}
where we used that $2^\ell\gp(2^{\ell})$ is asymptotically increasing in $\ell$. We obtain easily that the last sum remains bounded as $n\to\infty$.
In the end, we have a constant $c_8>0$ such that for $n\geq n_0$
\begin{equation}\label{bigell2}
  \bP\big(\tau_k = n,2^{\ell_0\vee\ell_{n,k}} < M_k \leq 2^{\ell_n-2} \big) \leq c_8 k(1-r_n)^k \bP(\tau_1=n).
\end{equation}

\medskip
{\bf Case 2.} To handle $\ell=\ell_n-1,\ell_n$ we have analogously to \eqref{maxbound}, for $n\geq n_0$
\begin{align}\label{topell}
 \bP\left( \tau_k=n, M_k>2^{\ell_n-2} \right) &\leq k(1-r_n)^{k-1} \max_{m\in (2^{\ell_n-2},2^{\ell_n}]} \bP\left(\tau_1=m \right) \notag\\
 &\leq c_9 k(1-r_n)^k \bP(\tau_1=n).
\end{align}

\medskip
{\bf Case 3.} We now deal with $\ell_0<\ell\leq \ell_{n,k}$.
We bound the last probability in \eqref{maxbound} using Lemma~\ref{lem:nottoolarge}. We obtain, analogously to \eqref{bigell}
\begin{align}
\label{smallell}
  \bP\big(\tau_k = n,2^{\ell_0} < M_k \leq 2^{\ell_{n,k}} \big) 
    &\leq 2c_{7} k \sum_{\ell=\ell_0+1}^{\ell_{n,k}\wedge\ell_n} (1-r_{2^\ell})^k \frac{\gp(2^{\ell})}{2^{\ell}}  
    \Big( \frac{c_3k 2^{\ell} \gp(2^{\ell}) }{n} \Big)^{n/2^{\ell+1}} \notag \\
  &\leq 2 c_7 k(1-r_n)^k \sum_{\ell=\ell_0+1}^{\ell_n} \frac{\gp(2^{\ell})}{2^{\ell}}  
    \Big( \frac12 \Big)^{n/2^{\ell+1}} \notag \\
  &\leq c_{10} k(1-r_n)^k\ \frac{\gp(2^{\ell_n+1})}{2^{\ell_n+1}}  \Big( \frac12 \Big)^{n/2^{\ell_n+1}} \notag\\
  &\leq c_{11} k(1-r_n)^k\frac{\gp(n)}{n} \notag\\
  &= c_{11} k(1-r_n)^k \bP(\tau_1=n).
\end{align}
Here the third inequality uses the fact that $n/2^{\ell_n+1} \geq 1/4$, and consequently the sum in the second line of \eqref{smallell} is of the same order as the $\ell=\ell_n$ term.

\medskip
{\bf Case 4.} Finally to handle $\ell\leq\ell_0$ we have, using Lemma \ref{lem:nottoolarge} and writing $m_0:=2^{\ell_0}$
\begin{align}
  \bP(\tau_k=n,M_k\leq 2^{\ell_0}) &\leq (1-r_{m_0})^k \bP\left( \sum_{i=1}^k \hat{G}_i^{(m_0)} =n \right)\notag\\
  &\leq (1-r_n)^k \left( \frac{1-r_{m_0}}{1-r_n} \right)^k \left( \min\Big\{ \frac{ c_3 m_0 \gp(m_0)k}{n} , 1 \Big\}\right)^{n/ m_0} \notag\\
  &\leq (1-r_n)^k \, e^{- c_{12} k} \left( \min\Big\{ \frac{ c_{13}k}{n} , 1 \Big\}\right)^{n/ m_0} .
\end{align}
Considering separately the cases $k\leq n/2c_{13}$ and $n/2c_{13}< k \leq n$, we conclude that there is some $c_{14}>0$ such that for $n$ large, 
\begin{equation}
\label{boundedell}
 \bP(\tau_k=n,M_k\leq 2^{\ell_0}) \leq (1-r_n)^k e^{-c_{14} n} \leq c_{15} k (1-r_n)^k\bP(\tau_1=n).
\end{equation}

\smallskip
Collecting \eqref{bigell2},\eqref{topell},\eqref{smallell} and \eqref{boundedell} concludes the proof of \eqref{uniformbound}.
\qed


\section{Large deviations: proof of Theorem \ref{largedev}}

Recall that $G_i=\tau_i-\tau_{i-1}$, and $\hat G_1^{(m)}, \hat G_2^{(m)},\ldots$ are i.i.d.\ with distribution $\bP(\tau_1\in\cdot\mid \tau_1\leq m)$.

\medskip
{\bf Proof of (i).} The second inequality is trivial, so we prove the first. Suppose $k\gp(n)\leq M$. 
Given $0<\gep<1$,
\[
  r_{\gep n} - r_{n} \sim \gp(n)\log\frac{1}{\gep} \quad \text{as } n\to +\infty,
\]
so for large $n$,
\begin{align}\label{smaller}
  \bP\left(\max_{i\leq k} \hat{G}_i^{(n)} \leq \gep n\right) = \left( 1 - \frac{ r_{\gep n} - r_{n} }{1-r_{n}} \right)^k
    \geq \exp\left( -2k\gp(n)\log\frac{1}{\gep} \right) \geq \gep^{2M}.
\end{align}
On the other hand, since $\bE[\hat G_1^{(m)}] \stackrel{m\to\infty}{\sim} m \gp(m)$, given $\gep>0$ we have for $n$ large enough
\begin{equation}\label{smallsum}
  \bP\left( \sum_{i=1}^k \hat G_i^{(\gep n)} \leq n \right) \geq 1- \frac{1}{n} \, k \bE\left( \hat G_1^{(\gep n)} \right)  \geq 1- 2 \gep k \gp(n) \geq 1-2\gep M.
\end{equation}
If $M\leq 1/3$, we apply \eqref{smallsum} with $\gep=1$:
\begin{equation}
\bP(\tau_k\leq n) \geq (1-r_{n})^k  \bP\left( \sum_{i=1}^n \hat G_i^{(n)} \leq n \right) \geq (1-r_n)^k (1-2M) \, .
\end{equation}
If $M>1/3$, we take $\gep = 1/4M$, and combining \eqref{smaller} with \eqref{smallsum}, we obtain for $n$ large enough
\[
\bP(\tau_k\leq n) \geq (1-r_{n})^k \bP\left(\max_{i\leq k} \hat{G}_i^{(n)} \leq \gep n\right) 
    \bP\left(\sum_{i=1}^k \hat G_i^{(\gep n)}  \leq n \right)  \geq  \frac12 \left( \frac{1}{4M}\right)^{2M}  (1-r_{n})^k .
\]

\smallskip

{\bf Proof of (ii)}. Define, for any $\gl>0$,
\[
  \nu(\lambda) := 1-\bE\left( e^{-\lambda \tau_1 } \right), 
\]
so $-\log(1-\nu(\cdot))$ is non-decreasing and strictly concave. Moreover, it is standard to obtain that
\begin{equation}\label{nunprops}
  \nu(\lambda)
\sim r \left( \frac{1}{\gl} \right) \to 0, \quad \text{and } \quad \nu'(\lambda) \sim   \frac{1}{\gl}\, \gp  \left( \frac{1}{\gl} \right)\to  +\infty \qquad \text{as } \gl\searrow 0.
  \end{equation}
We may view \eqref{phistarrate} as a combination of an upper and a lower bound, which we now prove.

\smallskip
\textit{Upper bound in \eqref{phistarrate}.}
Define
\[
 f_{n}(\lambda):= -n \lambda  - k \log(1-\nu(\gl));
\]
note the notation suppresses the dependence on $k$.
We will use the standard exponential bound
\begin{equation}\label{expb}
  \bP(\tau_k \leq n) = \bP\left( e^{-\lambda \tau_k } \geq e^{-\lambda n} \right) \leq e^{\lambda n} (1-\nu(\lambda))^k
    =e^{-f_{n}(\lambda)} \quad\text{for all } \lambda>0\, .
\end{equation}

Now, we define $\lambda_n>0$ by $f_n'(\lambda_n)=0$, or equivalently,
\begin{equation}\label{lambdandef}
  \frac{ \nu'(\lambda_n) }{ 1-\nu(\lambda_n) } = \frac{n}{k},
\end{equation}
so that $f_n$ achieves its (positive) supremum at $\lambda_n$.
Then $\gl_n\to 0$, since $n /k \to +\infty$.
Therefore, thanks to \eqref{nunprops}, we get that
\begin{equation}\label{nunprime2}
  \frac{n}{k} \simn \nu'(\lambda_n) \simn \frac{1}{\lambda_n} \gp\left( \frac{1}{\lambda_n} \right),
\end{equation}
which is equivalent to
\begin{equation}\label{lambdan}
 \frac{1}{\lambda_n} \simn \frac{n}{k}\gp^*\left( \frac{n}{k} \right).
\end{equation}
Then,  \eqref{nunprops} gives that $\nu(\gl_n) \simn r(1/\gl_n) \gg \gp(1/\gl_n)$, which with \eqref{nunprime2} shows that $n \lambda_n  \ll k \nu(\lambda_n)$.
In the end, we get  
\begin{equation}\label{fnsize2}
  f_n(\lambda_n) =(1+o(1)) k\nu(\lambda_n) \simn k\, r\left( \frac{1}{\lambda_n} \right)
    \simn k\, r\left( \frac{n}{k}\gp^*\left( \frac{n}{k} \right) \right).
\end{equation}
With \eqref{expb} this lets us conclude
\begin{equation}\label{upperbd1}
  \bP(\tau_k \leq n) \leq 
    \exp\left[ -(1+o(1))k\,  r\left( \frac{n}{k} \gp^*\left(\frac{n}{k}\right) \right) \right].
\end{equation}

\smallskip
\textit{Lower bound in \eqref{phistarrate}}
As is standard, we will obtain a corresponding lower bound using a tilted distribution.
Let $\gep>0$, and let $\tilde \lambda_n$ satisfy (analogously to \eqref{lambdandef})
\begin{equation}\label{lambdandef2}
  \frac{ \nu'(\tilde \lambda_n) }{ 1-\nu(\tilde \lambda_n) } = (1-\gep)\frac{n}{k}.
\end{equation}

%

Then, let $\tilde{\bP},\tilde{\bE},\mathrm{\tilde{V}ar}$ denote the probability, expectation and variance with respect to the tilted distribution of the i.i.d.\ sequence $(G_1,G_2,\dots)$ given by
\[
  \tilde{\bP}(G_1 \in \cdot) = \frac{ \bE\left( e^{-\tilde \lambda_n \tau_1} \ind_{\{\tau_1 \in \cdot \}} \right) }
    { \bE( e^{- \tilde\lambda_n \tau_1} ) }.
\]

We estimate
\begin{align}\label{usetilt}
  \bP(\tau_k \leq n)
  &\geq \frac{ \bE( e^{-\tilde \lambda_n \tau_k} ) }{ e^{-(1-2\gep) n \tilde \lambda_n } } \,
    \frac{ \bE\left( e^{-\tilde \lambda_n \tau_k} \ind_{\{\tau_k \in ((1-2\gep)n,n) \}} \right) }
    { \bE( e^{- \tilde\lambda_n \tau_k} ) } \notag\\
  &\geq \exp\bigg( (1-2\gep) n \tilde  \lambda_n  + k \log(1-\nu(\tilde\lambda_n)) \bigg) 
    \tilde{\bP} \big( \tau_k \in ((1-2\gep)n,n) \big).
\end{align}


Note that \eqref{lambdan} translates here as
\[\frac{1}{\tilde \gl_n} \simn (1-\gep)\frac{n}{k} \gp^*\left( \frac{n}{k}\right)\, ,\]
so that $\nu(\tilde\gl_n) \simn r(1/\tilde\gl_n) \simn\nu(\gl_n)$. As in \eqref{fnsize2}, we get that
\begin{equation}
 \bP(\tau_k \leq n) \geq  \exp\left[ -(1+o(1))k \,  r\left( \frac{n}{k} \gp^*\left(\frac{n}{k}\right) \right) \right] \times \tilde{\bP} \big( \tau_k \in ((1-2\gep)n,n) \big) \, ,
\end{equation}
and it only remains to show that the last probability converges to $1$ as $n\to+\infty$.

It is standard that 
\begin{equation}\label{titedmean}
  \tilde{\bE}\left( G_1 \right) = \frac{\nu'(\tilde\lambda_n)}{ 1-\nu(\tilde\lambda_n)} =  (1-\gep)\frac{n}{k},
\end{equation}
so we only need to show that $\mathrm{\tilde{V}ar} \left( G_1 \right) =o(n^2/k)$.
In fact, we have
\begin{equation}\label{tiltedvar2}
  \tilde{\bE}\left[ (G_1)^2 \right] = \frac{1}{1-\nu(\tilde \gl_n)} \sum_{j=1}^\infty j\gp(j)e^{- \tilde\lambda_n j}
    \simn \frac{1}{(\tilde \lambda_n)^2} \gp\left( \frac{1}{\tilde\lambda_n} \right) \simn \frac{1}{\tilde \gl_n}  (1-\gep) \frac{n}{k} ,
\end{equation}
where the last equivalence is a slight variant of \eqref{nunprime2}. Since $k\gp(n)\to \infty$, by a similar variant of \eqref{lambdan} we have 
\[  
  \frac{1}{\tilde \gl_n}  \gp\left( \frac{1}{\tilde \gl_n} \right) \sim (1-\gep)\frac{n}{k} \ll n\gp(n),
\]
and therefore $\tilde\gl_n^{-1} = o(n)$. With \eqref{tiltedvar2} this shows that indeed $\mathrm{\tilde{V}ar} \left( G_1 \right) =o(n^2/k)$.

\medskip
{\bf Proof of (iii)}.  The existence of $I(b)$ is standard, and its asymptotics as $b\to\infty$ simply follow from (ii).


\section{Reverse renewal theorems}

\subsection{Transient case, proof of Theorem \ref{thm:transient}}
\label{sec:transient}
Denote $p_{\infty}:=\bP(\tau_1 = +\infty) >0$.
We fix $\gep>0$, and $A$ large enough so $\bP(\tau_1>A) \in [p_{\infty}, p_{\infty}+\gep]$, and hence  $\bP(A<\tau_1<+\infty)\leq \gep$.
We then define the events
\[\mA_1=\{\tau\cap(0,A] =\emptyset\}\quad \text{ and }\quad \mA_2= \{ \tau\cap[n-A,n) =\emptyset\}.\]
We claim that if $n$ is large enough, 
\begin{gather}
(1-\gep) \bP(\tau_1\leq A)  \leq \bP(\mA_1^c | n\in\tau)\leq (1+\gep) \bP(\tau_1\leq A)\, , \notag \\
  (1-\gep) \bP(\tau_1\leq A) \leq \bP(\mA_2^c | n\in\tau)\leq (1+\gep) \bP(\tau_1\leq A)\, , \label{eq:Ac} \\
   (1-\gep) \bP(\tau_1\leq A)^2 \leq \bP(\mA_1^c \cap \mA_2^c | n\in\tau)\leq (1+\gep) \bP(\tau_1\leq A)^2 \, . \notag
\end{gather}
Indeed, we can write
\begin{equation}
\label{eq:A1A2c}
\bP(\mA_1^c \cap \mA_2^c | n\in\tau) = \sum_{i=1}^A \sum_{j=1}^A \bP(\tau_1 =i) \bP(\tau_1 =j) \frac{\bP(n- i-j \in\tau)}{ \bP(n\in\tau)} \, .
\end{equation}
Since $\bP(n\in\tau)$ is regularly varying, for large $n$, the last ratio is close to 1 uniformly in $i,j\leq A$, and the third line in \eqref{eq:Ac} follows.  The first two lines are proved similarly.

It follows from  \eqref{eq:Ac} that
\begin{align}\label{A1A2upper}
\bP(\mA_1 \cap \mA_2 | n\in\tau)&  = 1- \bP(\mA_1^c | n\in\tau) - \bP(\mA_2^c | n \in \tau ) + \bP(\mA_1^c \cap \mA_1^c  | n\in\tau) \notag \\
&\leq 1- 2 \bP(\tau_1\leq A) + \bP(\tau_1\leq A)^2 + 3\gep \notag\\
&\leq \bP(\tau_1>A)^2 + 3\gep \notag\\
&\leq (p_\infty + \gep)^2 + 3\gep\, .
\end{align}
Therefore for large $n$,
\begin{equation}
\label{tau1up}
\bP(\tau_1 =n) \leq \bP(\mA_1,\mA_2,n\in\tau) \leq \big( (p_{\infty}+\gep)^2 +3\gep\big) \, \bP(n\in\tau)\, . 
\end{equation}
Similarly to \eqref{A1A2upper}, $\bP(\mA_1 \cap \mA_2 | n\in\tau) \geq p_\infty^2 - 3\gep$ and hence
\begin{equation}\label{A1A2lower}
  \bP(\mA_1, \mA_2,n\in\tau) \geq (p_\infty^2 - 3\gep)  \bP(n\in\tau)\, . 
\end{equation}

\medskip
To turn this into a lower bound on $\bP(\tau_1=n)$, we show that conditionally on $\{\mA_1, \mA_2, n\in\tau\}$, it is very likely that $\tau_1=n$.
More precisely, we claim that there exists $c_{16}$ such that, for $n$ large,
\begin{equation} \label{claim}
  \bP(\tau_1\neq n , \mA_1, \mA_2,n\in\tau) \leq c_{16}\gep \bP(n\in\tau).
\end{equation}
With \eqref{A1A2lower}, this shows that 
\begin{equation}
\label{tau1low}
 \bP(\tau_1=n) = \bP(\tau_1= n , \mA_1, \mA_2,n\in\tau) \geq \big(p_\infty^2 - 3 \gep -  c_{16}\gep \big) \, \bP(n\in\tau) \, .
\end{equation}
Since $\gep$ is arbitrary, \eqref{tau1up} and \eqref{tau1low} complete the proof of Theorem \ref{thm:transient}.\\

\medskip
To prove \eqref{claim}, we write
\begin{align}
\label{decomp}
\bP(\tau_1\neq n , \mA_1, \mA_2,n\in\tau)
 \leq& \sum_{i=A+1}^{n/2} \bP( \tau_1 = i) \bP(n-i \in \tau) \notag \\
     & + \sum_{i=n/2+1}^{n-A-1} \sum_{ j= A+1}^{n-i} \bP(\tau_1 =i) \bP(\tau_1 =j) \bP(n-i-j \in\tau) \, .
\end{align}

For the first sum in \eqref{decomp}, since $\bP(n\in\tau)$ is regularly varying, there is a constant $c_{17}$ such that, provided that $n$ is large, $\bP(n-i\in\tau) \leq c_{17} \bP(n\in\tau)$ for every $i\leq n/2$. Hence
\begin{equation}
\label{part1}
\sum_{i=A+1}^{n/2} \bP( \tau_1 = i) \bP(n-i \in \tau) \leq c_{17} \bP(n\in\tau) \bP(A< \tau_1<+\infty) \leq c_{17}\, \gep\,  \bP(n\in\tau)\, .
\end{equation}

For the second sum in \eqref{decomp}, we use that $\bP(\tau_1 =i) \leq \bP(i\in\tau) \leq c_{17} \bP(n\in\tau)$ for $n$ large enough, since $i \in (n/2,n)$. Therefore,
\begin{align}
\label{part2}
\sum_{i=n/2+1}^{n-A-1} & \sum_{ j= A+1}^{n-i} \bP(\tau_1 =i) \bP(\tau_1 =j) \bP(n-i-j \in\tau) \notag \\
&\leq c_{17} \bP(n\in\tau) \bP(A<\tau_1 <+\infty) \times \sum_{k=0}^{+\infty} \bP(k\in\tau)
\leq \frac{c_{17}}{ p_{\infty}}\,  \gep \, \bP(n\in\tau)\, ,
\end{align}
and the proof of \eqref{claim} is complete.


\subsection{Recurrent case, proof of Theorem \ref{thm2}}
\label{sec:thm2}

The assumptions imply that the index of regular variation of $\bP(n\in\tau)$ must be $-1$.  Hence
we have $\bP(n\in\tau) = n^{-1} \ell_n$ with $\ell_n$ a slowly varying function. We can extend $U_n$ and $\ell_n$ to slowly varying functions $U(t)$ andf $\ell(t)$ defined on $[1,\infty)$

We now apply standard Tauberian arguments, in particular \cite[Corollary 1.7.3]{cf:BGT} which we use multiple times.

Set 
\[
 f(s):=\sum_{k=1}^{+\infty} s^k \bP(\tau_1=k),\quad u(s):= \sum_{k=0}^{+\infty} s^k \bP(k\in\tau), \quad |s|<1.
 \]
It is standard that, for $|s|<1$
\[u(s)(1-f(s)) = 1 \qquad \text{so }\qquad f'(s) = \frac{u'(s)}{u(s)^2} \, . \]
Since $U_n$ is slowly varying, we have $u(s) \sim U((1-s)^{-1})$ as $s\nearrow 1$. Similarly, since $u'(s)$ is the generating function of $(n+1) \bP(n+1\in\tau) \sim \ell_n$ as $n\to\infty$, we have $u'(s) \sim (1-s)^{-1}\ell((1-s)^{-1})$ as $s\uparrow 1$. We therefore conclude that
\[ f'(s) \sim \frac{1}{1-s} \frac{\ell(\tfrac{1}{1-s})}{U(\tfrac{1}{1-s})^2} \quad \text{as } s\nearrow 1 .\]
Since $\ell(t)/U(t)^2$ is slowly varying, it follows that $\sum_{k=0}^n k \bP(\tau_1=k) \sim  n \ell_n/U_n^2$ as $n\to\infty$.
This means that there is some $\epsilon_n $ decreasing to $0$ sufficiently slowly so that
\[ \sum_{k=(1-\epsilon_n) n}^{n} k \bP(\tau_1=k) \simn  n \epsilon_n \frac{\ell_n}{U_n^2}\ .\]
By \cite[Theorem 8.7.3]{cf:BGT} we have $U_n \simn \bP(\tau_1>n)^{-1}$, and \eqref{summed} follows. 
Equation \eqref{nosum} is an immediate consequence.

\subsection{Why not expect a stronger reverse renewal theorem?}
\label{sec:example}

In general, regular variation of $\bP(n\in\tau)$ (here with index of regular variation $-1$) does not imply regular variation of $\bP(\tau_1=n)$.  This shows that \eqref{nosum} cannot be true in general under the assumptions used to obtain \eqref{summed}. We give here only a description of an example, without proof details.

Let $\sigma$ be a recurrent renewal with inter-arrival distribution of form
\begin{equation}
\label{sigma}
\bP(\sigma_1 =n) = \gp(n)\, n^{-1} .
\end{equation}
Now, let $\tau_1$ be $2\sigma_1$ or $1$, with probability $1/2$ each:
\[\bP(\tau_1=1) =\tfrac12 \, , \quad \bP(\tau_1=2m) = \tfrac12 \bP(\sigma_1=m) \, , \quad \bP(\tau_1 = 2m-1) =0\, \quad \text{for } m\geq 1 .\]
Note that $r_n:=\bP(\tau_1 > n) \sim \frac12 \bP(\sigma_1 > n)$.
Then $\bP(\tau_1=n)$ is not regularly varying, but we will show that the gaps of length $1$ have a smoothing effect, and make $\bP(n\in\tau)$ regularly varying. More precisely, we claim that
\begin{equation}
\label{resulttau}
\bP(n\in\tau) \simn \frac{\gp(n)}{2r_n^2n} \simn \frac{\bP(\tau_1=2\lfloor \tfrac n2 \rfloor)}{2\bP(\tau_1>n)^2}\, ,
\end{equation}
where $\lfloor\cdot\rfloor$ denotes the integer part.


\medskip
\begin{proof}[Proof of \eqref{resulttau}.]
We choose $\theta_n,\lambda_n$ satisfying 
\[
r_n^{-1}\ll \theta_n \ll \gp(n)^{-1} \quad\text{and}\quad 1 \ll \lambda_n \ll r_n^{-1/2},
\]
and decompose $\bP(n\in\tau)$ into three sums:
\begin{equation}
\label{threesums}
\bP(n\in\tau)= \sum_{k\leq  (\gl_nr_n)^{-1} } \bP(\tau_k=n) + \sum_{ (\gl_nr_n)^{-1}< k 
  \leq \theta_n} \bP(\tau_k=n) + \sum_{k > \theta_n} \bP(\tau_k=n)\, .
\end{equation}
We will show that the main contribution comes from the middle sum, see \eqref{middlesum}, the first and last sum being negligible.

\medskip
{\bf Middle sum.} We introduce $X_k$ the number of gaps of length $1$ in the first $k$ gaps of $\tau$.
For $(\gl_nr_n)^{-1}< k\leq \theta_n$, note that $ \gl_n \leq  k \gl_n^2 r_n \ll k$ and $k\ll n$, and write
\begin{align}
\label{goodterms}
\bP(\tau_k=n) = \bP\big( X_k&- \tfrac{k}{2}\in (- k \gl_nr_n^{1/2},  k \gl_nr_n^{1/2})  \, ;\, \tau_k=n \big) \notag\\
&+ \bP\big( | X_k-\tfrac{k}{2} | \geq k \gl_nr_n^{1/2} \, ;\, \tau_k=n )\, .
\end{align}
The last probability is small. Indeed, there is a constant $c_{33}$ such that 
\[
\bP\big( |X_k- \frac{k}{2}|  \geq k \gl_nr_n^{1/2}\big) \leq e^{-c_{33} \gl_n^2r_n k}\quad \text{for all } k\geq 1;
\]
conditioning on $X_k$ we therefore get that
\begin{equation}
\bP\big( | X_k-\tfrac{k}{2} | \geq k \gl_nr_n^{1/2} \, ;\, \tau_k=n )\leq e^{-c_{33} \gl_n^2r_n k}  \sup_{1\leq j\leq k}\ \sup_{ \frac{n-k}{2}\leq m\leq \frac{n}{2}} \bP(\sigma_j=m) \, .
\end{equation}
Here the sups are over all possible values of $j=k-X_k$ and $m=(n-X_k)/2$.
Applying \eqref{uniformbound} we see that for $n$ large, for all $m\geq (n-k)/2 \geq n/4$ and $j\leq k$, we have 
\[
\bP(\sigma_j=m) \leq c_{34}k \bP(\sigma_1=n).
\] 
Since $e^{-c_{33} \gl_n^2r_n k} = o(1)\, e^{-kr_n}$ as $n\to\infty$, uniformly in middle-sum values of $k$, we get that
\begin{equation}
\label{Xkdev}
\bP\big( | X_k-\tfrac{k}{2} | \geq k \gl_nr_n^{1/2} \, ;\, \tau_k=n \big) = o(1) k e^{-k r_n} \frac{\gp(n)}{n}\quad\text{as } n\to\infty  \, ,
\end{equation}
with the $o(1)$ uniform over middle-sum values of $k$.

For the first probability on the right in \eqref{goodterms}, we use \eqref{locallimit}. Uniformly for $j$ in the interval $\tfrac{k}{2} + (- k \gl_nr_n^{1/2},  k \gl_nr_n^{1/2})$ with $j\equiv n-k\mod 2$, and for middle-sum values of $k$ (which satisfy $k \gp(n)\stackrel{n\to\infty}{\to}0$ and $k\to +\infty$),  we have $j \sim k/2$ and $k\ll n$, so
\begin{align}
\bP \Big(\sigma_j = \frac{n-k+j}{2} \Big) &= (1+o(1)) k \big( 1- 2 r_{(n-k+j)/2} \big)^j \frac{\gp\big( n/2 \big)}{n/2} \notag\\
&= (1+o(1))\, 2k e^{-k r_n(1+o(1))}\,  \frac{\gp(n)}{n},
\end{align}
since $\bP(\sigma_1 >n) \simn 2 r_n$.  Therefore, since 
\[
\bP\Big(  X_k- \tfrac{k}{2}\in (- k \gl_nr_n^{1/2},  k \gl_nr_n^{1/2}) \, ;\, X_k \equiv n \mod 2 \Big) \to \frac12\quad\text{as } n\to\infty\, ,
\]
conditioning again on $X_k$ we get that for middle-sum values of $k$,
\begin{equation}
\label{Xkgood}
\bP\Big(X_k- \tfrac{k}{2}\in (- k \gl_nr_n^{1/2},  k \gl_nr_n^{1/2}) \, ;\, \tau_k=n \Big) = (1+o(1))\, ke^{-k r_n(1+o(1))}\,  \frac{\gp(n)}{n} ,
\end{equation}
with the $o(1)$ uniform over middle-sum values of $k$.

Summing \eqref{Xkdev} and \eqref{Xkgood}, we obtain straightforwardly that
\begin{equation}
\label{middlesum}
\sum_{ (\gl_nr_n)^{-1}< k \leq \theta_n} \!\!\! \bP(\tau_k=n) = \!\!\! \sum_{ (\gl_nr_n)^{-1}< k \leq \theta_n} \!\!\! (1+o(1)) \, ke^{-k r_n(1+o(1))}\, \frac{\gp(n)}{n}   =(1+o(1))\, \frac{\gp(n)}{2 r_n^2 n} .
\end{equation}

\medskip
We are therefore left with showing that the two other sums in \eqref{threesums} are negligible.  

\smallskip
{\bf First sum.} Using \eqref{uniformbound} and \eqref{sigma} and conditioning once more on $X_k$, we get that there exists a constant $c_{35}$ such that for $n$ large enough, for any $ k\leq (\gl_nr_n)^{-1}$, 
\[\bP(\tau_k=n) \leq \sup_{1\leq j\leq k} \sup_{\frac{n-k}{2} \leq m\leq \frac{n}{2}}  \bP(\sigma_j =m) \leq  c_{35} \, k \frac{\gp(n)}{n} \, , \]
which gives 
\begin{equation}
\label{firstsum}
\sum_{k\leq  (\gl_nr_n)^{-1} } \bP(\tau_k=n) \leq \frac{c_{35}}{(\gl_nr_n)^2} \frac{\gp(n)}{n} =o(1) \frac{\gp(n)}{r_n^2 n} \, .
\end{equation}

\smallskip
{\bf Last sum.}
Similarly to \eqref{Xkdev}-\eqref{Xkgood} but using \eqref{uniformbound} in place of \eqref{locallimit}, we obtain that there exists $c_{36}$ such that for all $\theta_n< k\leq n/2$,
\[\bP(\tau_k = n) \leq c_{36} k e^{-  k r_n(1+o(1))} \frac{\gp(n)}{n} \leq k e^{-kr_n/2 } \frac{\gp(n)}{n} ,\]
the last inequality being valid for $n$ large, since $k r_n \to +\infty$.\\

For $ k \in ( n/2,n]$, we use that
\[\bP(\tau_k = n)  \leq (1-r_n)^k \leq e^{- nr_n/2 } \, .\]
Since $\theta_n \gg r_n^{-1}$, we therefore obtain that
\begin{equation}
\label{lastsum}
\sum_{k >\theta_n } \bP(\tau_k = n) \leq \sum_{k>\theta_n}  k e^{-kr_n/2} \frac{\gp(n)}{n} + n  e^{-nr_n/2 }  =o(1) \frac{\gp(n)}{r_n^2 n} \, . 
\end{equation}
This completes the proof of \eqref{resulttau}.
\end{proof}

\medskip
{\bf Acknowledgments:} The authors are grateful to V. Wachtel for bringing the work of Nagaev \cite{cf:Nagaev} to their attention.  We also thank the anonymous reviewer, which pointed out several references, and suggested a shorter proof of Theorem \ref{thm2}.

\end{document}